\documentclass[11pt]{amsart}
\usepackage{amssymb}
\usepackage[latin1]{inputenc}
\usepackage{pst-fun}
\usepackage{bm}

\newtheorem{lemma}{Lemma}[section]

\newtheorem{theorem}[lemma]{Theorem}
\newtheorem{remark}[lemma]{Remark}
\newtheorem{proposition}[lemma]{Proposition}
\newtheorem{corollary}[lemma]{Corollary}
\newtheorem{definition}[lemma]{Definition}
\newtheorem{example}[lemma]{Example}

\newcommand{\codim}{{\rm{codim\hspace{2pt}}}}

\newcommand{\e}{\varepsilon}

\newcommand{\bR}{{\mathbb R}}
\newcommand{\bC}{{\mathbb C}}

\begin{document}
\title{Fibrations structure and degree formulae for Milnor fibers}

\begin{abstract}
In this survey, we remind some fibrations structure theorems (also called Milnor's fibrations) recently proved in the real and complex case, in the local and global settings. We give several Poincaré-Hopf type formulae which relates the Euler-Poincaré characteristic of these fibers (also called Milnor's fibers) and indices (topological degree) of appropriated vector fields defined on spheres of radii small or big enough. 
\end{abstract}

\author{N. Dutertre, R. Ara\'ujo dos Santos, Y. Chen and A. Andrade}

\address{N. Dutertre: Aix-Marseille Universit\'e, CNRS, Centrale Marseille, I2M, UMR 7373,
13453 Marseille, France.}
\email{nicolas.dutertre@univ-amu.fr}

\address{R. N. Ara\'ujo dos Santos, Y. Chen and A. A. Espirito Santos: Departamento de Matem\'{a}tica, Instituto de Ci\^{e}ncias Matem\'{a}ticas e de Computa\c{c}\~{a}o, Universidade de S\~{a}o Paulo - Campus de S\~{a}o Carlos, Caixa Postal 668, 13560-970 S\~ao Carlos, SP, Brazil}

\email{rnonato@icmc.usp.br}
\email{yingchen@icmc.usp.br}
\email{andrade@icmc.usp.br}

\maketitle

\section{Introduction}

In this survey, we provide a brief review of some extensions of the real and complex Milnor fibration's theorems for isolated and non-isolated singularities and we give some Poincaré-Hopf type formulae which gives a connection between the Euler-Poincaré characteristic of the (regular) Milnor's fibers and the index of a gradient vector field. The paper is organized as follows. In Section 2, we first remind some fibration structure theorems which extend the very well known complex and real local Milnor's fibrations stated in \cite{Mi}. In the sequel, we remind Milnor's formula for complex isolated singularity (Poincaré-Hopf type formula) and we state several extensions of such formula which has been developed recently for functions and mappings with non-isolated singularities, in the complex and real case. 

In Section 3, we try to follow the line of development of Section 1, i.e., we also start introducing and showing some fibrations's Theorems, but now in the global case (at infinity), and in the sequel we show several Poincaré-Hopf type formulae in this setting.

\section{Fibration structure and degree formulae - Local setting}

\subsection{Milnor's fibrations structure}

For a germ of complex holomorphic function $f: (\bC^{n+1}, 0) \to (\bC,0)$, in \cite[Theorem 4.8, page. 43]{Mi} Milnor proved that there exists a small enough $\epsilon_{0}>0$ such that for all $0<\epsilon \leq \epsilon_{0}$ the projection

\begin{equation}\label{Milcomplex}
\frac{f}{\|f\|}: S^{2n+1}_\e \setminus \{f=0\} \to S^{1}
\end{equation}

\noindent is a smooth locally trivial fibration.

Later, it was proved by Lê D. Tr\'ang that the restriction map

$$f_{|}:\mathring{B}_{\epsilon}^{2n+2}\cap f^{-1}(S^{1}_{\delta})\to S^{1}_{\delta}$$

\noindent is also the projection of a locally trivial smooth fibration and that such a fibration is fiber-equivalent to the fibration (\ref{Milcomplex}), where $\mathring{B}_{\epsilon}^{2n+2}$ denotes the interior of the closed ball $B_{\epsilon}^{2n+2}.$ 

For a real analytic map germ $F:(\mathbb{R}^{n},0) \to (\mathbb{R}^{p},0)$, $n > p \geq 2$ with isolated singularity at origin, Milnor also proved in \cite[Theorem 11.2, page. 97]{Mi} that, for all $\epsilon >0$ small enough, there exists $0<\eta \ll \epsilon$ such that the restriction map
\begin{equation*}
\displaystyle{F_{|}: B_{\epsilon}^{n}\cap F^{-1}(S^{p-1}_{\eta})} \to S^{p-1}_{\eta}
\end{equation*}
\noindent is the projection of a locally trivial smooth fibration (called fibration in the ``Milnor tube"), with fiber 
$\mathcal{M}_{F}^{T}:=(F_{|})^{-1}(y).$ He also constructed a convenient vector field in $B_{\epsilon}^{n}\setminus V,$ where $V=F^{-1}(0)$, and used its flow to push the Milnor tube $B_{\epsilon}^{n}\cap F^{-1}(S^{p-1}_{\eta})$ to the sphere $S_{\epsilon}^{n-1},$ but keeping the boundary $S_{\epsilon}^{n-1}\cap F^{-1}(S^{p-1}_{\eta})$ fixed. Now, since $0$ is an isolated critical point of $F,$ it is not hard to see that is possible to extend this fibration to get a smooth locally trivial fibration  $$S_{\epsilon}^{n-1}\setminus {\rm Lk}^{0}(V) \to S^{p-1},$$
where ${\rm Lk}^{0}(V):=V\cap S_{\epsilon}^{n-1}$ is called the link of the zero set $V$ at origin.


As Milnor pointed out, in general we can not expect that the projection of this last fibration will be given by the canonical one 
$\frac{F}{\|F\|}.$ In \cite{AT} the authors provide a characterization of such a fibration in the more general case of non-isolated singularity (see also \cite{RA,Ar, CSS}) as follows.

\begin{theorem}

Let $F:(\mathbb{R}^{n},0) \to (\mathbb{R}^{p},0)$, $n > p \geq 2$ be an analytic mapping and $\Sigma_{F}$ its singular locus i.e., the set of points where the gradients $\nabla f_1,\ldots,\nabla f_p$ are linearly dependent. Assume that $\Sigma_{F}\cap V \subseteq \{0\}.$ The following two statements are equivalent:

\begin{enumerate}

\item for all $\epsilon >0$ small enough the projection $\frac{F}{\|F\|}: S_{\epsilon}^{n-1}\setminus {\rm Lk}^{0}(V) \to S^{p-1}$ is a smooth locally trivial fibration.

\item for all $\epsilon >0$ small enough the projection $\frac{F}{\|F\|}: S_{\epsilon}^{n-1}\setminus {\rm Lk}^{0}(V) \to S^{p-1}$ is a smooth submersion.

\end{enumerate} 
\end{theorem}

We remark that, this equivalence is not so trivial because if the link is not empty such a projection mapping is never proper.

The Milnor tube fibration for non-isolated singularity was established by D. Massey in \cite{Ma} as follows.

Let $F=(f_1,\ldots,f_p): (\mathbb{R}^n,0) \rightarrow (\mathbb{R}^p,0)$ be an analytic map, $2\leq p \leq n-1,$ and $\rho$ be the square of the distance function to the origin and denote by $M(F)$ the set of critical points of the pair $(F,\rho)$, i.e., the set of points where the gradients $\nabla \rho,\nabla f_1,\ldots,\nabla f_p$ are linearly dependent. This set is called the {\it Milnor set} and it also plays an important role in the study of asymptotical behavior of polynomial mappings at infinity, as the reader will see in the next section.

It follows by definition that $\Sigma_F\subseteq M(F).$

\begin{definition}\cite{Ma}\label{defMilnorCond}Given $F$ and $\rho $ as above.
\vspace{0.2cm}
\begin{enumerate}
\item We say that $F$ satisfies Milnor's condition (a) at the origin, if $\Sigma_F \subset V$ in a neighborhood of the origin.
\item We say that $F$ satisfies Milnor's condition (b) at the origin, if $0$ is isolated in $V \cap \overline{M(F)\setminus V}$ in a neighborhood of the origin, where the notation $\overline{X}$ means the topological closure of the space $X$.
\end{enumerate}
\end{definition}

\begin{remark}
It follows from Definition \ref{defMilnorCond} the equivalence:

The mapping $F$ satisfies Milnor's condition $(b)$ at the origin if and only if for each $\epsilon>0$ small enough, there exist $\delta>0$, $0<\delta \ll \epsilon $ such that the restriction map $F_|:S_{\epsilon}^{n-1}\cap F^{-1}({B^{p}_{\delta}\setminus \{0\}})\to B^{p}_{\delta}\setminus \{0\}$ is a smooth submersion (and onto, if the link of $F^{-1}(0)$ is not empty).
\end{remark}

We say that $\epsilon >0$ is a {\it Milnor radius for $F$ at origin}, provided that $B_{\epsilon}^{n}\cap {(\overline{\Sigma_{F}-V}})= \varnothing $, and $B_{\epsilon}^{n}\cap {V\cap(\overline{M(F)\setminus V})}\subseteq \{0\}$, where
$B_{\epsilon}^{n}$ denotes the closed ball in $\mathbb{R}^{n}$ with radius $\epsilon$.

Consequently, under Milnor's conditions $(a)$ and $(b)$, we can conclude that for all regular values close to the origin the respective fibers into the closed $\epsilon$-ball are smooth and transverse to the sphere $S^{n-1}_{\epsilon}.$

\begin{theorem}[\cite{Ma}, Theorem 4.3, page 284]\label{massey}
Let $F:(f_1,\ldots,f_p): (\mathbb{R}^n,0) \rightarrow (\mathbb{R}^p,0)$ satisfying Milnor's conditions (a) and (b) and $\epsilon_{0}>0$ be a Milnor's radius for $F$ at origin. Then, for each $0<\epsilon \leq \epsilon_{0}$, there exist $\delta,$ $0<\delta \ll \epsilon ,$ such that
\begin{equation}
F_{|}:B_{\epsilon}^{n}\cap F^{-1}(B^{p}_{\delta}\setminus \{0\})\to B^{p}_{\delta}\setminus \{0\}
\end{equation}
is the projection of a smooth locally trivial fiber bundle.

\end{theorem}

\proof (Idea)
Since $\epsilon_{0}>0$ is a Milnor's radius for $F$ at origin, we have that $\Sigma_{F}\cap B_{\epsilon_{0}}^{n} \subset V\cap B_{\epsilon_{0}}^{n}$.
It means that, for all $0<\epsilon \leq \epsilon_{0}$ the map $F_|: \mathring{B}_{\epsilon }^{n}\setminus V \to \mathbb{R}^{p}$ is a smooth submersion in the open ball $\mathring{B}_{\epsilon }^{n}$.

From the Milnor condition $(b)$, and the remark above, it follows that: for each $\epsilon$ there exists $\delta $, $0<\delta \ll \epsilon $, such that
\begin{equation*}
 \displaystyle{ F_|: S_{\epsilon}^{n-1} \cap F^{-1}(B_{\delta}^{p}-\{0\}) \to B_{\delta}^{p}-\{0\}}
\end{equation*}
is a submersion on the boundary $\displaystyle{S_{\epsilon}^{n-1}}$ of the closed ball $\displaystyle{ B_{\epsilon}^{n}}$.

Now, combining these two conditions we have that, for each
$\epsilon $, we can choose $\delta $ such that
\begin{equation*}
 F_|: B_{\epsilon}^{n}\cap F^{-1}(B_{\delta}^{p}-\{0\}) \to  B_{\delta}^{p}-\{0\}
\end{equation*}
is a proper smooth submersion. Applying the Ehresmann Fibration Theorem for the manifold with boundary $B_{\epsilon}^{n}$, we get that it is a smooth locally trivial fibration. \endproof

\begin{corollary} (Milnor's fibration in the tube)\label{Miltube}
Let $F:(f_1,\ldots,f_p): (\mathbb{R}^n,0) \rightarrow (\mathbb{R}^p,0)$ satisfying Milnor's conditions (a) and (b) and $\epsilon_{0}>0$ be a Milnor's radius for $F$ at origin. Then, for each $0<\epsilon \leq \epsilon_{0}$, there exists $\delta,$ $0<\delta \ll \epsilon ,$ such that
\begin{equation}
F_{|}:B_{\epsilon}^{n}\cap F^{-1}(S^{p-1}_{\delta})\to S^{p-1}_{\delta} \label{Milnortube}
\end{equation}
is the projection of a smooth locally trivial fiber bundle.
\end{corollary}

\begin{example}\label{R1}
Let $f:(\mathbb{C}^{n},0)\to (\mathbb{C},0)$ be a holomorphic function germ, then it satisfies the Milnor conditions $(a)$ and $(b).$

In fact, Milnor's condition $(a)$ and $(b)$ can be seen as an application of \L ojasiewicz's inequality which states that, in a small neighborhood of the origin, there are constants $C>0$ and $0<\theta<1$ such that
$$|f(x)|^{\theta}\leq C\|\nabla f(x)\|.$$
So, Milnor's condition $(a)$ follows. In \cite{HL}, page 323, Hamm and Lê proved that the \L ojasiewicz inequality implies Thom $a_{f}$-condition for a Whitney $(a)$ stratification of $V.$ Therefore, Milnor's condition $(b)$ follows.
\end{example}

\begin{example}
Let $F:(\mathbb{R}^{n},0)\to (\mathbb{R}^{p},0)$ be an analytic map-germ with an isolated singular point at origin. Then, Milnor's conditions $(a)$ and $(b)$ above hold. In fact, Milnor's condition $(b)$ follows since the zero locus is transversal to all small spheres.
\end{example}

In \cite{AT, ACT} the authors considered the following definition.

\begin{definition}\label{d:booksinginf} \cite{ACT}.
We say that the pair $(K, \theta)$ is a \textit{higher open book structure with singular binding}  on an analytic manifold $M$ of dimension $m-1 \ge p \ge 2$ if $K\subset M$ is a singular real subvariety of codimension $p$ and $\theta : M\setminus K \to S^{p-1}_1$ is a locally trivial smooth fibration such that $K$ admits a neighbourhood $N$ for which the restriction $\theta_{|N\setminus K}$ is the composition $N\setminus K  \stackrel{h}{\to} B^p \setminus \{ 0\} \stackrel{\pi}{\to} S^{p-1}_1$ where $h$ is a locally trivial fibration and $\pi (s,y)=s/\|s\|.$

In such a case, one says that the \textit{singular fibered subvariety} $K$ is the \textit{binding} and that the (closures of) the fibers of $\theta$ are the \textit{pages} of the \textit{open book}.
\end{definition}

Let us denote by $M(\frac{F}{\|F\|})$ the singular locus of the pair $(\frac{F}{\|F\|},\rho),$ where $\rho(x)=x_{1}^{2}+\cdots +x_{n}^{2}$ is the square of distance function to the origin. Inspired by Milnor's conditions in \cite{Ma}, in the paper \cite{ACT} the authors proved some theorems about fibration structures on spheres of radii small enough for non-isolated singularity, as follows. 

\begin{theorem}[\cite{ACT}, Theorem 1.3, page 819]
Let $F:(\mathbb{R}^{n},0)\to (\mathbb{R}^{p},0)$ be a real analytic mapping germ. Assume that Milnor's condition $(b)$ holds and $\codim_{\mathbb{R}}V=p.$ If for all $\epsilon>0$ small enough the set $M(\frac{F}{\|F\|})$ is empty, then the pair $({\rm Lk}^{0}(V),\frac{F}{\|F\|})$ is a higher open book structure on the spheres $S_{\epsilon}^{n-1},$ with singular binding. In particular, the projection 

\begin{equation}\label{fibspheres}
\frac{F}{\|F\|}: S_{\epsilon}^{n-1}\setminus {\rm Lk}^{0}(V) \to S^{p-1}
\end{equation}

\noindent is a smooth locally trivial fibration. 
\end{theorem}

It seems natural to ask the following question: how does the fibers of fibrations (\ref{Milnortube}) and (\ref{fibspheres}) relates with each other, once both fibrations exist ?

Denote by $\mathcal{M}_{F}^{T}$ and $\mathcal{M}_{\frac{F}{\|F\|}}^{S}$ the fibers of fibrations (\ref{Milnortube}) and (\ref{fibspheres}), respectively. In \cite{CSS} the authors considered a condition called $d$-regularity and assuming also that $F$ satisfies a Thom $a_{F}$-condition, they showed that both fibrations are fiber-equivalent. We should point out that, it is pretty clear that Thom $a_{F}$-condition along $V$ imply Milnor's condition (b). But, the converse in not true as the reader can check in \cite{ACT}, section 5.3, page 827, in an example provided by A. Parusinsky. 

More recently, just assuming Milnor's condition (b) and using a quite different machinery, the authors in \cite{DAAC} proved a similar result to the one stated in \cite{CSS}. This result was proved first by M. Oka in \cite{Ok} for non-degenerate and convenient mixed polynomial germs.

\begin{proposition}[\cite{DAAC}, section 5.0.1, Proposition 5.2] Let $F=(f_1,\ldots,f_p): (\mathbb{R}^n,0) \rightarrow (\mathbb{R}^p,0)$ be a real analytic map germ satisfying Milnor's conditions $(a)$ and $(b)$. If  $M(\frac{F}{\|F\|})$ is empty (as a germ of set), then
the two fibers $\mathcal{M}_{F}^{T}$ and $\mathcal{M}_{\frac{F}{\|F\|}}^{S}$ are homotopy equivalent. 
\end{proposition}

At this point, we would like to invite the interested reader to look at section 5.0.1 of \cite{DAAC} for further details.

\subsection{A degree formula in the complex case}

Let $f:(\mathbb{C}^{n+1},0)\to (\mathbb{C},0)$ be a germ of holomorphic function. In the case that $0\in \mathbb{C}^{n+1}$ is an isolated critical point of $f,$ Milnor considered the topological degree, denote by ${\rm deg}_{0} \nabla f,$ of the normalized gradient vector field
\begin{equation}
\epsilon \frac{\nabla f}{\| \nabla f \|} :S_\epsilon ^{2n+1}  \rightarrow S_{\epsilon}^{2n+1},
\end{equation}

\noindent where $\nabla f = (\frac{\partial f}{\partial z_{0}}, \cdots, \frac{\partial f}{\partial z_{n}})$. Let us denote by $\mu_{f}:={\rm deg}_{0} \nabla f.$ 

This number was named later as ``Milnor number" and it has several important meanings. For instance, it is known that the regular fiber $F_{\theta}=(\frac{f}{\Vert f \Vert})^{-1}(\theta)$ has the homotopy type of a bouquet of $n$-dimensional spheres and the number of spheres in the bouquet is $\mu_{f}.$ Therefore, the Euler-Poincaré characteristic of the (regular) fiber satisfies the following {\it Poincaré-Hopf formula}
\begin{equation}\label{Milfor}
\chi(F_{\theta})=1+(-1)^n\mu_{f}.
\end{equation}

In the search of topological and analytical invariants of complex and real singularities, this type of Poincaré-Hopf formula became a starting point of several others formulae, for isolated and non-isolated singularities, in the local and global settings. Below we will discuss briefly some of them in the local real setting.

\subsection{Degree formulae in the real case} In \cite{Kh} Khimshiashvili considered a germ of real analytic function $f: (\mathbb{R}^n,0) \rightarrow (\mathbb{R},0)$ with isolated critical point at the origin, and proved that
\begin{equation}\label{Kh}
\chi \big(f^{-1}(\delta) \cap B_{\epsilon}^{n} \big)=1-\hbox{sign}(-\delta)^n  \hbox{deg}_0 \nabla f,
\end{equation}

\noindent where $0 < \vert \delta \vert \ll \varepsilon \ll 1$ is a regular value, $B_{\epsilon}^{n}$ stands for the closed ball centered at the origin, $\nabla f$ is the gradient vector field of $f$ and $\mathrm{deg}_0 \nabla f$ is the topological degree of the mapping
$$\epsilon \frac{\nabla f}{\| \nabla f \|} :S_\epsilon ^{n-1} \rightarrow S_{\epsilon}^{n-1}.$$

The formula above was extended for mappings with isolated singularity as follows.

Let $F=(f_{1},\cdots, f_{p}): (\mathbb{R}^{n},0)\to (\mathbb{R}^{p},0) $ be an analytic mapping with isolated critical point at the origin. Then, each coordinate function $f_{i},$ for $i=1,\cdots, p,$ also has isolated critical point at the origin. Using these notations, formulae (\ref{Milfor}) and (\ref{Kh}) above were extended in the following way:

\begin{theorem}\cite[Theorem 1.3, page. 210]{ADD}\label{euler}
Given $F:(\mathbb{R}^{n},0) \to (\mathbb{R}^{p},0)$, $n > p \geq 1$,
as above, the following holds true:
\begin{itemize}
\item[(i)] If $n$ is even, then
$\chi (\mathcal{M}_{F}^{T}) = 1-\mathrm{deg}_0 \nabla f_1$.
Moreover, we have
$$\mathrm{deg}_0 \nabla f_{1}= \mathrm{deg}_0 \nabla f_{2}
= \cdots= \mathrm{deg}_0 \nabla f_{p}.$$
\item[(ii)] If $n$ is odd, then
$\chi (\mathcal{M}_{F}^{T}) = 1$. Moreover, we have
$\mathrm{deg}_0 \nabla f_{i}=0$ for $i=1, 2, \ldots, p$.
\end{itemize}
\end{theorem}

\begin{example}[Case of isolated singularity of holomorphic functions]
Given a germ of holomorphic function $f: (\bC^{n+1}, 0) \to (\bC,0)$ with isolated critical point at origin, we consider the complex variable $z_{j}=x_{j}+iy_{j},$ for $1\leq j \leq n+1$ and $f=(P,Q):(\mathbb{R}^{2n+2},0)\to (\mathbb{R}^{2},0)$ where $P=\Re(f)$ and $Q=\Im(f)$ is the real and imaginary part of $f,$ respectively. By Cauchy-Riemann equations, we have that
$$\nabla P(x,y)=(\frac{\partial P(x,y)}{\partial x_{1}}-i\frac{\partial P(x,y)}{\partial y_{1}}, \ldots, \frac{\partial P(x,y)}{\partial x_{n+1}}-i\frac{\partial P(x,y)}{\partial y_{n+1}} ) ,$$

which we can be identified with
$$\nabla P(x,y)=(\frac{\partial P(x,y)}{\partial x_{1}}, -\frac{\partial P(x,y)}{\partial y_{1}}, \ldots, \frac{\partial P(x,y)}{\partial x_{n+1}},-\frac{\partial P(x,y)}{\partial y_{n+1}}) .$$

Denote by $\displaystyle{H(x,y):=(\frac{\partial P(x,y)}{\partial x_{1}},\frac{\partial P(x,y)}{\partial y_{1}}, \ldots, \frac{\partial P(x,y)}{\partial x_{n+1}},\frac{\partial P(x,y)}{\partial y_{n+1}})}$ and by ${\rm deg}_{0} \nabla H$ the topological degree of the map $\epsilon \frac{\nabla H}{\| \nabla H \|} :S_{\epsilon}^{2n+1} \rightarrow S_{\epsilon}^{2n+1},$ for all $\epsilon>0$ small enough. Now, it is easy to see that
$${\rm deg}_{0} \nabla P =(-1)^{n+1}.{\rm deg}_{0} \nabla H $$ and so Milnor's formula (\ref{Milfor}) follows from Theorem \ref{euler}, item i).
\end{example}

It seems natural to search for similar result for non-isolated singularities. This was done in \cite{DA} and we will explain below the main strategy used to get an analogous formula.

Let us consider $F=(f_1,\ldots,f_p): (\mathbb{R}^n,0) \rightarrow (\mathbb{R}^p,0),$ $1\leq p \leq n-1,$ an analytic mapping.
Consider $l \in \{1,\ldots, p\}$ and $I= \{i_1,\ldots,i_l\}$ an $l$-tuple of pairwise distinct elements of $ \{1,\ldots,p\}$ and let us denote by $f_I$ the mapping $(f_{i_1},\ldots,f_{i_l}): (\mathbb{R}^n,0) \rightarrow (\mathbb{R}^l,0)$. Suppose that $F$ satisfies Milnor's condition $(a)$ at the origin. Then, we have
$$\Sigma_{f_I} \subset \Sigma_F \subset F^{-1}(0) \subset f_{I}^{-1}(0),$$
and so by definition the map $f_I$ also satisfies Milnor's condition (a) at the origin.

It is clear that $\Sigma_{f_I,\rho} \subset \Sigma_{F,\rho}.$ In \cite{DA} it was proved the following key result:

\begin{lemma}\cite{DA} Assume that $F$ satisfies Milnor's conditions $(a)$ and $(b)$ at the origin. Then, for all $l \in \{1,\ldots, p\}$ and $I= \{i_1,\ldots,i_l\} \subset \{1,\ldots,p\}$, the maps $f_I: (\mathbb{R}^n,0) \rightarrow (\mathbb{R}^l,0)$ satisfies Milnor's conditions $(a)$ and $(b)$.
\end{lemma}
\proof See \cite[Lemma 4.1, page 7]{DA}.
\endproof

\begin{corollary}
There exists $\epsilon_{0}>0$ such that, for all $l \in \{1,\ldots, p\}$ and $I= \{i_1,\ldots,i_l\} \subset \{1,\ldots,p\}$, the maps $f_I:  (\mathbb{R}^n,0) \rightarrow (\mathbb{R}^l,0)$ have $\epsilon_{0}$ as a Milnor's radius. Therefore, for all $2\leq l \leq p$ there exist the Milnor fibrations (see Theorem \ref{massey} and Corollary \ref{Miltube}) for the maps $f_I : (\mathbb{R}^n,0) \rightarrow (\mathbb{R}^l,0).$
\end{corollary}

In \cite{DA} several formulae were proved connecting the Euler-Poincaré characteristics of the links of the sets $f_{I}^{-1}(0)$ and of the Milnor fiber $\mathcal{M}_{F}^{T}$ for an analytic map germ $F$ under Milnor's conditions $(a)$ and $(b)$. We remind below some of them.

As above, let us choose $l \in \{1,\ldots,p \}$ and an $l$-tuple $I=\{i_1,\ldots,i_l\}$ of pairwise distinct elements of $\{1,\ldots,p\}$. We write $J=\{i_1,\ldots,i_{l-1}\}$.
We also denote by ${\rm Lk}^{0}(V_{I})$ (resp. ${\rm Lk}^{0}(V_{J})$) the local link of the zero-set of $f_I$ (resp. $f_J$). If $l=1$ then $J= \emptyset$ and we put $f_J=0$.

\begin{proposition}\label{CharL1}
We have:
$$\chi({\rm Lk}^{0}(V_{J}))-\chi({\rm Lk}^{0}(V_{I}))=(-1)^{n-l} 2 \chi(\mathcal{M}_{F}^{T}).$$
\end{proposition}
\proof See \cite[Proposition 7.1, page 10]{DA}. \endproof

\begin{corollary} \label{milnorformula}
Let $j \in \{1,\ldots,p\}$. If $n$ is even, then we have $\chi({\rm Lk}^{0}(V_{\{j\}}))= 2 \chi(\mathcal{M}_{F}^{T})$ and if $n$ is odd, then we have $\chi({\rm Lk}^{0}(V_{\{j\}}))=2-2\chi(\mathcal{M}_{F}^{T})$.
\end{corollary}
\proof We apply the previous proposition to the case $l=1$. In this case, if $n$ is even then $\chi({\rm Lk}^{0}(V_{J}))=0$ and if $n$ is odd then $\chi({\rm Lk}^{0}(V_{J}))=2$. \endproof

Let us explain how to apply this result in order to get a degree formula for $\chi(\mathcal{M}_{F}^{T}).$

Given $h_{1},\cdots , h_{s}:U\subseteq \mathbb{R}^{n}\to \mathbb{R}$ analytic functions, defined in an open neighborhood of the origin $U,$ with $h_{1}(0)=\cdots =h_{s}(0)=0$, let $h(x)=h_{1}(x)^2+\cdots +h_{s}(x)^2.$ Of course, the following equality of analytic sets holds $$\{x\in U:~ h_{1}(x)=\cdots =h_{s}(x)=0\}=\{x\in U:~ h(x)=0\}.$$ Szafraniec in \cite{Sz1} considered the function $g(x)=h(x)-c\rho(x)^{k},$ where $c>0$ and $k$ an integer, and showed that for $k$ large enough the function $g$ has an isolated singular point at the origin. Moreover, he proved that for all small radius $\epsilon $, the following Poincaré-Hopf type formula holds true:
\begin{equation}\label{Szafr}
\chi(\{x\in S_{\epsilon}^{n-1}: ~ h(x)=0\})=1-\deg_{0} \nabla g.
\end{equation}

From Szafraniec's equation above and Corollary \ref{milnorformula}, we can state a Poincaré-Hopf type formula for local non-isolated singularities as follows:

\begin{proposition}
Given $F=(f_1,\ldots,f_p): (\mathbb{R}^n,0) \rightarrow (\mathbb{R}^p,0),$ $1\leq p \leq n-1,$ an analytic map germ satisfying Milnor's conditions $(a)$ and $(b).$ There exist $c>0$ and an integer $k$ such that if $g(x)=f_{1}^{2}(x)-c.\rho(x)^{k}$ then we have, 

\begin{enumerate}
\item [(i)] if $n$ is even, then $\chi(\mathcal{M}_{F}^{T})=\frac{1}{2}(1-\deg_{0} \nabla g);$
\item [(ii)] if $n$ is odd, then $\chi(\mathcal{M}_{F}^{T})=\frac{1}{2}(1+\deg_{0} \nabla g)$.
\end{enumerate}
\end{proposition}

\proof  It follows from Corollary \ref{milnorformula} and Szafraniec's results.  \endproof

\section{Fibration structure and degree formulae - Global setting}

\subsection{(Global Milnor) Fibration's structure}

It is well known that, for a polynomial holomorphic function $f:\bC^{n+1} \to \bC,$ there exists a finite set of complex values called the bifurcation set, denoted by $B_{f} \subset \mathbb{C},$ such that the restriction map $f_{|}: \mathbb{C}^{n+1} \setminus f^{-1}(B_{f}) \to \mathbb{C} \setminus B_{f}$ is a smooth locally trivial fibration. If one takes a closed disc in the target space, $D_{R}\subset \mathbb{C}$ with $R>0$ big enough in such a way that $B_{f}$ is included inside its interior, this fibration induces the so-called {\it global monodromy fibration of $f,$} 

\begin{equation}
f_{|}: \mathbb{C}^{n+1}\cap f^{-1}(S_{R}^{1})\to S_{R}^{1}\label{nemethi} 
\end{equation}

\noindent which is also a smooth locally trivial fibration.

In \cite{NZ} Nemethi and Zaharia considered a kind of regularity condition called {\it semitame condition} which control the asymptotic behavior of the fibers at infinity. Following Milnor's method for the local case, they proved that for all radii $R$ big enough, denote by $R\gg 1,$ the canonical (Milnor) projection 

\begin{equation}
\displaystyle{\frac{f}{\|f\|}}: S^{2n+1}_{R} \setminus \{f=0\} \to S^{1} \label{globalmilnor}
\end{equation}

\noindent is a smooth locally trivial fibration. But, as the authors explained in the introduction in \cite{NZ} (see Section 4 in \cite{NZ}, for further results), one can check in Broughton's example $f(x,y)=x+x^{2}y$ that the fiber of the Milnor projection (\ref{globalmilnor}) is the thrice punctured 2-sphere and the generic fiber (\ref{nemethi}) is the twice punctured 2-sphere. So, these two fibrations can not be fiber-equivalent.

In the real case, this fibration structure has been also approched in \cite{ACT1} and more recently in a more general way by the present authors in \cite{DAAC}. Below we give the main ideas and results of \cite{ACT1}.

Let $F=(f_1,\ldots,f_p): \mathbb{R}^n \rightarrow \mathbb{R}^p$, $2 \le p \le m-1$, be a polynomial mapping, $V = F^{-1}(0)$ and for $R\gg 1$, denote by ${\rm Lk}^{\infty}(V):=V\cap S_{R}^{n-1}$ the link of $V$ at infinity. Let us denote by $\rho : \mathbb{R}^{n}\to \mathbb{R},$ $\rho(x)=x_{1}^{2}+\cdots +x_{n}^{2},$ and so $S_{R}^{n-1}=\rho^{-1}(R).$

Consider the (global Milnor's) projection $\frac{F}{\Vert F \Vert}: S_{R}^{n-1}\setminus {\rm Lk}^{\infty}(V) \rightarrow S^{p-1}$ and, as in the local case, also denote by: 

\begin{enumerate}

  \item $\Sigma_F$ the set of critical points of $F.$
	
  \item $M(\frac{F}{\|F\|})$ the set of critical points of the pair $(\frac{F}{\|F\|},\rho)$ inside $\mathbb{R}^n \setminus V.$
  
  \item $M(F)$ the critical locus of the pair $(F,\rho).$
	
\end{enumerate}

\vspace{0.2cm}


With such a definition, the authors in \cite{ACT1} proved a fibration structure on spheres of radii big enough, as follows.

Following the definition stated in \cite{ACT1}, we will consider the following conditions:

Condition (A): $\overline{M(F)\setminus V}\cap V$ in bounded in $\mathbb{R}^{n}.$

Condition (B): $M(\frac{F}{\|F\|})$ is bounded in $\mathbb{R}^{n}.$

\begin{theorem}\label{t:sing}
Let $F : \bR^{n} \to \bR^p$ be a real polynomial map such that $\codim_\bR V = p$ at infinity\footnote{For the zero locus $V = F^{-1}(0)\subset \bR^{n}$, we say that ``$V$ has codimension $p$ at infinity'' if for any radius $R\gg 1$, $\codim_\bR V \setminus B_R = p$, in the sense that every irreducible component of $V$ has this property.} and assume that Condition (A) holds. Then we have the equivalence:

\begin{enumerate} 
\item [(i)] $({\rm Lk}^{\infty}(V), \frac{F}{\|F\|})$ is an open book decomposition of $S_R^{n-1}$ with singular binding, independent of the high enough radius $R\gg 1$ up to $C^\infty$ isotopy.
\item [(ii)]  Condition (B) holds.
\end{enumerate}
\end{theorem}

\begin{proof}
\noindent
We follow closely the proof given in \cite{ACT1}. We should say that the same idea of this proof was used in \cite{DAAC}, but to prove a fibration theorem in a more general setting. 

Firstly, see that Condition (B) holds if and only if for all $R\gg 1$ the projection mapping
$$\frac{F}{\|F\|}: S_{R}^{n-1}\setminus {\rm Lk}^{\infty}(V)\rightarrow S^{p-1}$$
is a submersion.

By the Curve Selection Lemma, Condition $(A)$ means that:
there exists $\delta>0$ small enough and a closed disc $B_{\delta}^{p}$ centered at origin such that
$$ F_{|}: F^{-1}(B_{\delta}^{p}-\{0\})\cap S_{R}^{n-1}\rightarrow B_{\delta}^{p}-\{0\}$$
is a surjective proper and smooth submersion.

\vspace{0.2cm}

The implication (i) $\Rightarrow$ (ii) is trivial, because if we suppose that $\frac{F}{\|F\|}: S_{R}^{n-1}\setminus {\rm Lk}^{\infty}(V) \rightarrow S^{p-1}$ is a locally trivial fibration, then it is a submersion and so Condition (B) follows.

\vspace{0.2cm}

Let us prove the converse. We can assume that ${\rm Lk}^{\infty}(V)$ is not empty: in fact, in the case it is empty, Condition (A) is always satisfied and by Condition (B), we have that $\frac{F}{\|F\|}: S_{R}^{n-1} \rightarrow S^{p-1}$ is a proper submersion and so onto, since $S^{p-1}$ is connected. Hence (i) follows by Ehresmann's theorem.

In the case ${\rm Lk}^{\infty}(V)$ is not empty, the projection $\frac{F}{\|F\|}: S_{R}^{n-1}\setminus {\rm Lk}^{\infty}(V) \rightarrow S^{p-1}$ is not proper and we can not use Ehresmann's theorem. Nevertheless, the mapping $ F_{|}: F^{-1}(B_{\delta}^{p}-\{0\})\cap S_{R}^{n-1} \rightarrow B_{\delta}^{p}-\{0\}$ is a
surjective proper submersion. So, by Ehresmann's theorem, the mapping

\begin{equation}\label{eq:fib}
F_{|}: F^{-1}(B_{\delta}^{p}-\{0\})\cap S_{R}^{n-1} \rightarrow B_{\delta}^{p}-\{0\}
\end{equation}
is a locally trivial fibration. Now, we can compose it with the radial projection $\pi_{1}: B_{\delta}^{p}-\{0\} \rightarrow S^{p-1}, $ $\pi_{1} (y)=\frac{y}{\|y\|},$ to get that
\begin{equation}\label{eq:fib1}
\frac{F}{\|F\|}: F^{-1}(B_{\delta}^{p}-\{0\})\cap S_{R}^{n-1} \rightarrow S^{p-1}
\end{equation}
is a locally trivial fibration, as well.

Fibration (\ref{eq:fib}) yields that the map

\begin{equation}\label{eq:fib 3}
\frac{F}{\|F\|}: F^{-1}(S_{\delta}^{p-1})\cap S_{R}^{n-1} \rightarrow S^{p-1}
\end{equation}
is also a locally trivial fibration and therefore surjective.

This implies that
\begin{equation}\label{eq:fib2}
\frac{F}{\|F\|}:  S_{R}^{n-1}\setminus F^{-1}(\mathring{B}_{\delta}^{p})\rightarrow  S^{p-1}
\end{equation}
is surjective and proper, where $\mathring{B}_{\delta}^{p}$ denotes the open disk for some $0<\delta \ll \frac{1}{R}$. So, by using Condition (B), the mapping $\frac{F}{\|F\|}$ is a smooth submersion. Therefore, we have a locally trivial fibration by Ehresmann's theorem.

Now we can glue fibrations \eqref{eq:fib1} and \eqref{eq:fib2} along the common boundary $F^{-1}(S_{\delta}^{p-1})\cap S_{R}^{n-1},$ using fibration (\ref{eq:fib 3}), to get the item $(i)$. \end{proof}

\begin{remark}
As mentioned above, inspired by (the local) Milnor's conditions and the global conditions introduced in \cite{ACT1}, the present authors in \cite{DAAC} introduced two conditions that produce open book structures on $C^{2}$ semi-algebraic manifolds, generalizing and unifying this way open book structures on spheres of big or small radii.
\end{remark}

\subsection{A degree formula in the complex case}

In \cite{Br, HZ, NZ, NZ1} the authors considered a polynomial holomorphic function $f:\bC^{n+1} \to \bC$ under the so-called tame, $M$-tame and semitame conditions. We do not intend to give more details in this direction, but we remind briefly the definitions of these conditions and give some topological results concerning the (global) Milnor fibers.

If one fixes coordinates $z=(z_{0},\cdots, z_{n}),$ then one can consider the following quotient algebra 

$$\mathcal{Q}_{f}=\displaystyle{\frac{\mathbb{C}[z_{0},\ldots, z_{n}]}{(\frac{\partial f}{\partial z_{0}},\ldots, \frac{\partial f}{\partial z_{n}})}},$$ 

\noindent where $\mathbb{C}[z_{0},\ldots, z_{n}]$ denotes the ring of polynomial holomorphic functions and $(\frac{\partial f}{\partial z_{0}},\ldots, \frac{\partial f}{\partial z_{n}})$ denotes the ideal spanned by the partial derivatives of $f.$

It is well known that $f$ has only finitely many isolated singular points if and only if the complex dimension of the quotient algebra $\dim_{\mathbb{C}}\mathcal{Q}_{f}$ is finite. This dimension is called the total Milnor number and it is denoted as in the local case by $\mu_{f}.$ As in the local case, the number $\mu_{f}$ is equal to the topological degree, denote by ${\rm deg}_{\infty} \nabla f,$ of the normalized gradient vector field
\begin{equation}
R \frac{\nabla f}{\| \nabla f \|} :S_{R} ^{2n+1}  \rightarrow S_{R}^{2n+1}.
\end{equation} 

Denote by $\overline{\nabla}(f)=\displaystyle{(\overline{\frac{\partial f}{\partial z_{0}}},\ldots, \overline{\frac{\partial f}{\partial z_{n}}})}$ and the so-called (complex) Milnor set by $M(f)=\{z\in \mathbb{C}^{n+1}; \overline{\nabla}(f)=\lambda z\}$ for some $\lambda \in \mathbb{C}.$


Under the notations above, consider the following condition $(*):$ for every sequence $\{z_{k}\} \subset M(f)$ 
such that $\|z_{k}\|\to +\infty,$  $\overline{\nabla}(f)(z_{k})\to 0.$

\begin{definition} 

We say that a polynomial function $f$ under condition $(*)$ is:
\begin{enumerate}
\item [(i)] $M$-tame if  $|f(z_{k})|\to +\infty.$

\item [(ii)] semitame if  $|f(z_{k})|\to c,$ for some $c\in \mathbb{C},$ then $c=0.$
\end{enumerate}
\end{definition}

Of course, if $f$ is $M$-tame, then it is semitame. But, the converse is not true, see for instance \cite{NZ}.

\begin{theorem} [\cite{NZ, NZ1} see also \cite{HZ}]

Let $f:\bC^{n+1} \to \bC$ be a $M$-tame holomorphic polynomial function. Then, the following holds true:

\begin{enumerate}

\item [(i)] the bifurcation $B_f$ is equal to the singular locus $\Sigma_{f}.$

\item [(ii)] for each value $w\in \mathbb{C}\setminus B_{f},$ the (typical) fiber $f^{-1}(w)$ has the homotopy type of a bouquet of $\mu_{f}$ $n$-dimensional spheres.

\item [(iii)] the fibration (\ref{globalmilnor}) is fiber-equivalent to the global monodromy fibration of $f.$


\end{enumerate} 

\end{theorem}

So, it follows from item $[(ii)]$ above that for such $M$-tame class of polynomial we still have the following Poincaré-Hopf type formula $$\chi(F_{\theta})=1+(-1)^{n}\mu_{f},$$ where $F_{\theta}$ is the fiber of fibration (\ref{globalmilnor}). 

\subsection{Degree formulae in the real case}
Inspired by the formulae proved in local case and in the global complex case, we present below a degree formula for the Euler characteristic of the real Milnor fiber in a big sphere. The strategy is the same as in the local case. First we need to establish global versions of Szafraniec's theorem \cite{Sz1}. Note that the results below are new.

\subsubsection{On the link at infinity of a closed semi-algebraic set} We establish several degree formulae for the Euler characteristic of the link at infinity of a closed semi-algebraic set. If $X \subset \mathbb{R}^n$ is a semi-algebraic set, then we denote by ${\rm Lk}^\infty(X)$ its link at infinity.

Let $f : \mathbb{R}^n \rightarrow \mathbb{R}$ be a $C^2$-semi-algebraic function. We assume that $0$ is a critical value of $f$. We denote by $\Sigma_f= (\nabla f)^{-1}(0)$ the set of critical of points of $f$.
We remark that $\Sigma_f \cap f^{-1}(0)$ may not be compact.

Let $\omega(x)= 1 + \frac{1}{2}(x_1^2+\cdots+x_n^2)$. Note that $\nabla \omega(x)=x$ and $\omega(x) \ge 1$.
Let 
$$\Gamma_{f,\omega} = \left\{ x \in \mathbb{R}^n \ \vert \ {\rm rank} [\nabla f(x), \nabla \omega(x) ] < 2 \right\}.$$
We have $\Sigma_f \subset \Gamma_{f,\omega}$. 
\begin{lemma}\label{first}
There is $k \in \mathbb{N}$ such that for all $x \in \Gamma_{f,\omega} \setminus f^{-1}(0)$, 
$$\vert f(x) \vert > \frac{1}{\omega(x)^k},$$
for $\vert x \vert \gg 1$.
\end{lemma}
\proof Let $r_0$ be the greatest critical value of $\omega$. We set $\tilde{S}_r = \omega^{-1}(r)$. Let $\beta : ]r_0,+\infty [ \rightarrow \mathbb{R}$ be defined by 
$$\beta (r) = {\rm inf} \left\{ \vert f(x) \vert \ \vert \ x \in \tilde{S}_r \cap (\Gamma_{f,\omega} \setminus f^{-1}(0)) \right\}.$$
The function $\beta$ is semi-algebraic. Furthermore $\beta >0$ because for $r > r_0$, $f_{\vert \tilde{S}_r}$ has a finite number of critical values. Thus the function $\frac{1}{\beta}$ is also semi-algebraic. Hence there exist $r_1 \ge r_0$ and $k_0 \in \mathbb{N}$ such that $\frac{1}{\beta} < r^k$, for $r \ge r_1$ and $k \ge k_0$. This implies that $\beta (r) > \frac{1}{r^k}$ for $r \ge r_1$ and $k \ge k_0$. We can conclude that for $r \ge r_1$ and $k \ge k_0$,
$$\vert f(x) \vert > \frac{1}{\omega(x)^k},$$
for $ x \in \tilde{S}_r \cap ( \Gamma_{f,\omega} \setminus f^{-1}(0) ).$ \endproof

Let $g_-(x) =f(x)-\frac{1}{\omega(x)^k}$. Note that $\Gamma_{f,\omega}=\Gamma_{g_-,\omega}$.
\begin{lemma}
For $R \gg 1$, $\chi ( \{ g_- \le 0 \} \cap \tilde{S}_R ) = \chi ( \{ f \le 0 \} \cap \tilde{S}_R )$.
\end{lemma}
\proof Let $R \gg 1$ be such that for all $x \in \Gamma_{f,\omega} \setminus f^{-1}(0) \cap \{\omega(x) \ge R\}$,  $\vert f(x) \vert > \frac{1}{\omega(x)^k}$.
Let $N_f^\le = \{ x \in \tilde{S}_R \ \vert \ f(x) \le 0 \}$ and $N_{g_-}^\le = \{ x \in \tilde{S}_R \ \vert \ g_-(x) \le 0 \}$. For $x \in \tilde{S}_R$, we have 
$$g_-(x) \le 0 \Leftrightarrow f(x)-\frac{1}{R^k} \le 0 \Leftrightarrow f(x) \le \frac{1}{R^k},$$
and so $N_f^\le \subset N_{g_-}^\le$. Furthermore if $0 < f(x) \le \frac{1}{R^k}$ then $x \notin \Gamma_{f,\omega} \setminus f^{-1}(0)$ and therefore $\{ f(x) \le \frac{1}{R^k} \} \cap \tilde{S}_R$ retracts by deformation to $\{ f(x) \le 0 \} \cap \tilde{S}_R$. 
We get the result. \endproof

\begin{corollary}
We have $\chi ({\rm Lk}^\infty (\{g_- \le 0 \}) ) = \chi  ({\rm Lk}^\infty (\{f \le 0 \}) ) $.
\end{corollary}

$\hfill \Box$

Let $g_+(x) =f(x)+\frac{1}{\omega(x)^k}$. Note that $\Gamma_{f,\omega}=\Gamma_{g_+,\omega}$.
\begin{lemma}
For $R \gg 1$, $\chi ( \{ g_+ \ge 0 \} \cap \tilde{S}_R ) = \chi ( \{ f \ge 0 \} \cap \tilde{S}_R )$.
\end{lemma}
\proof Let $R \gg 1$ be such that for all   $x \in \Gamma_{f,\omega} \setminus f^{-1}(0) \cap \{\omega(x) \ge R \}$, $\vert f(x) \vert > \frac{1}{\omega(x)^k}$. 
Let $N_f^\ge = \{ x \in \tilde{S}_R \ \vert \ f(x) \ge 0 \}$ and $N_{g_+}^\ge = \{ x \in \tilde{S}_R \ \vert \ g_+(x) \ge 0 \}$. For $x \in \tilde{S}_R$, we have 
$$g_+(x) \ge 0 \Leftrightarrow f(x)+\frac{1}{R^k} \ge 0 \Leftrightarrow f(x) \ge -\frac{1}{R^k},$$
and so $N_f^\ge \subset N_{g_+}^\ge$. Furthermore if $0  > f(x) \ge -\frac{1}{R^k}$ then $x \notin \Gamma_{f,\omega} \setminus f^{-1}(0)$ and therefore $\{ f(x) \ge - \frac{1}{R^k} \} \cap \tilde{S}_R$ retracts by deformation to $\{ f(x) \ge 0 \} \cap \tilde{S}_R$. 
We get the result. \endproof

\begin{corollary}
We have $\chi ({\rm Lk}^\infty (\{g_+ \ge 0 \}) ) = \chi  ({\rm Lk}^\infty (\{f \ge 0 \}) ) $.
\end{corollary}

$\hfill \Box$

\begin{lemma}
The sets $(\nabla g_-)^{-1}(0) \cap \{g_-=0 \}$ and $(\nabla g_+)^{-1}(0) \cap \{g_+=0 \}$ are compact.
\end{lemma}
\proof If $x \in (\nabla g_-)^{-1}(0) \cap \{g_-=0 \}$ then $x \in \Gamma_{g_-, \omega}= \Gamma_{f,\omega}$ and $f(x)=\frac{1}{\omega(x)^k} \not= 0$. So $x \in \Gamma_{f,\omega} \setminus f^{-1}(0)$. But there exists $R \ge r_1$ such that 
$\vert f(x) \vert > \frac{1}{\omega(x)^k}$ if $\rho(x) \ge R$ and $ x \in \Gamma_{f,\omega} \setminus f^{-1}(0)$. We conclude that $$\nabla ( g_-)^{-1}(0) \cap \{g_-=0 \} \subset \tilde{B}_R,$$ where $\tilde{B}_R = \{\omega \le R \}$. \endproof

We make the assumption that $f(0) > 1$. 
\begin{lemma}
We have $g_-(0) >0$ and $g_+(0) >0$.
\end{lemma}
\proof We have $g_-(0)=f(0)-1$ and $g_+ (0)=f(0) +1$. \endproof
Let $H_-$ and $H_+ : \mathbb{R}^{n+1} \rightarrow \mathbb{R}^{n+1}$ be defined by 
$$H_-(\lambda,x)= ( \lambda x + \nabla g_-, g_-) \hbox{ and } H_+(\lambda,x)= ( \lambda x + \nabla g_+, g_+).$$

\begin{lemma}
The sets $H_-^{-1}(0)$ and $H_+^{-1}(0)$ are compact.
\end{lemma}
\proof See \cite{Du1}, Lemma 5.4 and Remark 5.11. \endproof

Let $\delta_-$ be a small regular value of $g_-$. Let $R_- \gg 1$ be such that $\{ g_-=0 \} \cap \tilde{B}_{R_-}$ is a retract by deformation of $\{ g_-=0\}$, $\{ g_-\le 0 \} \cap \tilde{B}_{R_-}$ is a retract by deformation of $\{ g_-\le 0\}$ and $\{ g_-\ge 0 \} \cap \tilde{B}_{R_-}$ is a retract by deformation of $\{ g_- \ge 0\}$. 

\begin{proposition}
If $n$ is even, we have
$$ \chi ( \{ g_- = \delta_- \} \cap \tilde{B}_{R_-} ) = {\rm deg}_\infty H_-.$$
If $n$ is odd, we have
$$ \chi ( \{ g_- \ge \delta_- \} \cap \tilde{B}_{R_-} ) - \chi ( \{ g_- \le \delta_- \} \cap \tilde{B}_{R_-} )= 1-{\rm deg}_\infty H_-.$$
\end{proposition}
\proof Apply equalities {\bf (A)} and {\bf (D)} in \cite{Du1}, page 332, and use the fact that $g_{-}(0)>0$. \endproof

If $n$ is even, we have
$$ \chi ( \{ g_- = \delta_- \} \cap \tilde{B}_{R_-} )  = \frac{1}{2}  \chi ( \{ g_- = \delta_- \} \cap \tilde{S}_{R_-} ) =  \chi ( \{ g_- \le \delta_- \} \cap \tilde{S}_{R_-} ) =$$
$$\chi ( \{ g_- \le 0 \} \cap \tilde{S}_{R_-} ) = \chi ({\rm Lk}^\infty ( \{ g_- \le 0 \})) .$$

If $n$ is odd, we have
$$ \chi ( \{ g_- \ge  \delta_- \} \cap \tilde{B}_{R_-} ) - \chi ( \{ g_- \le  \delta_- \} \cap \tilde{B}_{R_-} ) =$$ $$ \frac{1}{2}  \left[ \chi ( \{ g_- \ge  \delta_- \} \cap \tilde{S}_{R_-} )-\chi ( \{ g_- \le  \delta_- \} \cap \tilde{S}_{R_-} ) \right].$$
Since,
$$\chi(\tilde{S}_{R_-})=2 = \chi ( \{ g_- \ge  \delta_- \} \cap \tilde{S}_{R_-} ) + \chi ( \{ g_- \le  \delta_- \} \cap \tilde{S}_{R_-} ),$$
we find that $\chi ( \{ g_- \ge  \delta_- \} \cap \tilde{S}_{R_-} )=2-\chi ( \{ g_- \le  \delta_- \} \cap \tilde{S}_{R_-} )$ and that 
$$ \chi ( \{ g_- \ge  \delta_- \} \cap \tilde{B}_{R_-} ) - \chi ( \{ g_- \le  \delta_- \} \cap \tilde{B}_{R_-} ) = 1- \chi ( \{ g_- \le  \delta_- \} \cap \tilde{S}_{R_-} )=$$
$$1-\chi ( \{ g_- \le 0 \} \cap \tilde{S}_{R_-} )=1 -\chi( {\rm Lk}^\infty (\{g_- \le 0 \})).$$
\begin{corollary}
We have 
$$\chi ( {\rm Lk}^\infty ( \{g_- \le 0 \})) = \chi ( {\rm Lk}^\infty ( \{f \le 0 \})) = \hbox{\rm deg}_\infty H_-.$$
\end{corollary}

$\hfill \Box$

Let $\delta_+$ be a small regular value of $g_+$. Let $R_+ \gg 1$ be such that $\{ g_+=0 \} \cap \tilde{B}_{R_+}$ is a retract by deformation of $\{ g_+=0\}$, $\{ g_+\le 0 \} \cap \tilde{B}_{R_+}$ is a retract by deformation of $\{ g_+\le 0\}$ and $\{ g_+\ge 0 \} \cap \tilde{B}_{R_+}$ is a retract by deformation of $\{ g_+ \ge 0\}$. 

\begin{proposition}
If $n$ is even, we have
$$ \chi ( \{ g_+ = \delta_+ \} \cap \tilde{B}_{R_+} ) = {\rm deg}_\infty H_+.$$
If $n$ is odd, we have 
$$ \chi ( \{ g_+ \ge \delta_+ \} \cap \tilde{B}_{R_+} ) - \chi ( \{ g_+ \le \delta_+ \} \cap \tilde{B}_{R_+} )= 1-{\rm deg}_\infty H_+.$$
\end{proposition}
\proof Apply equalities {\bf (A)} and {\bf (D)} in \cite{Du1}, page 332, and use the fact that $g_+(0)>0$. \endproof

\begin{corollary}
If $n$ is even, we have
$$\chi ({\rm Lk}^\infty (\{ f \ge 0 \})) ={\rm deg}_\infty H_+.$$
If $n$ is odd, we have
$$\chi ({\rm Lk}^\infty (\{ f \ge 0 \})) =2-{\rm deg}_\infty H_+.$$
\end{corollary}
\proof If $n$ is even, we apply the same proof as for ${\rm Lk}^\infty( \{ f \le 0 \})$. 

If $n$ is odd, we remark that 
$$ \chi ( \{ g_+ \ge \delta_+ \} \cap \tilde{B}_{R_+} ) - \chi ( \{ g_+ \le \delta_+ \} \cap \tilde{B}_{R_+} )=$$
$$ \chi ( \{ g_+ \ge \delta_+ \} \cap \tilde{S}_{R_+} ) - 1= \chi({\rm Lk}^\infty ( \{ f \ge 0 \})) -1= 1-{\rm deg}_\infty H_+.$$
\endproof

\begin{corollary}
If $n$ is even, we have
$$\chi ({\rm Lk}^\infty (\{f=0 \})) = {\rm deg}_\infty H_- + {\rm deg}_\infty H_+.$$
If $n$ is odd, we have
$$\chi ({\rm Lk}^\infty (\{f=0 \})) = {\rm deg}_\infty H_- - {\rm deg}_\infty H_+.$$
\end{corollary}
\proof Use the equality
$$\chi(S^{n-1})= \chi ({\rm Lk}^\infty (\{f \ge 0 \})) + \chi ({\rm Lk}^\infty (\{f  \le 0 \})) -\chi ({\rm Lk}^\infty (\{f=0 \})) ,$$
and the previous corollaries. \endproof

\underline{Application:} Let $X \subset \mathbb{R}^n$ be a closed semi-algebraic set. There exists a $C^2$-semi-algebraic function $f : \mathbb{R}^n \rightarrow \mathbb{R}$ such that $X=f^{-1}(0)$ and $f \ge 0$. By a change of origin if necessary, we can assume that $f(0)>1$.
In this case, $g_+ >0$ and so ${\rm deg}_\infty H_+ =0$.
\begin{corollary}
We have $\chi({\rm Lk}^\infty(X))={\rm deg}_\infty H_-$.
\end{corollary}

$\hfill \Box$

Now we return to the general case. The functions $g_-$ and $g_+$ are semi-algebraic but not polynomial even if $f$ is a polynomial. Let $G_-$ and $G_+$ be defined by
$$G_-(x)= \omega(x)^k g_-(x)=\omega(x)^k f(x)-1,$$ and $$G_+(x)= \omega(x)^k g_+(x)= \omega(x)^k f(x)+1.$$
If $f$ is a polynomial then so are $G_-$ and $G_+$.  Furthermore, $\{G_-=0\}=\{g_-=0\}$, $\{G_-\le 0\}=\{g_- \le 0\}$, $\{G_- \ge 0\}=\{g_- \ge0\}$ and $\{G_+=0\}=\{g_+=0\}$, $\{G_+\le 0\}=\{g_+ \le 0\}$, $\{G_+ \ge 0\}=\{g_+\ge0\}$.
\begin{lemma}
The sets $(\nabla G_-)^{-1}(0) \cap \{G_-=0 \}$ and $(\nabla G_+)^{-1}(0) \cap \{G_+=0 \}$ are compact.
\end{lemma}
\proof We have 
$$\nabla G_-(x) = k \omega(x)^{k-1} g_-(x)\nabla\omega (x)+ \omega(x)^k \nabla g_-(x).$$
If $G_-(x)=0$ then $g_-(x)=0$ and so $\nabla G_-(x) = \omega(x)^k \nabla g_-(x)$. We see that $(\nabla G_-)^{-1}(0) \cap \{G_-=0 \}=(\nabla g_-)^{-1}(0) \cap \{g_-=0 \}$, which is compact. \endproof

Let $L_-$ and $L_+ : \mathbb{R}^{n+1} \rightarrow \mathbb{R}^{n+1}$ be defined by 
$$L_-(\lambda,x)= ( \lambda x + \nabla G_-, G_-) \hbox{ and } L_+(\lambda,x)= ( \lambda x + \nabla G_+, G_+).$$

\begin{lemma}
The sets $L_-^{-1}(0)$ and $L_+^{-1}(0)$ are compact.
\end{lemma}
\proof See \cite{Du1}, Lemma 5.4 and Remark 5.11. \endproof

\begin{corollary}
We have
$$ \chi ( {\rm Lk}^\infty ( \{f \le 0 \})) = \hbox{\rm deg}_\infty L_-,$$
$$ \chi ( {\rm Lk}^\infty ( \{f \ge 0 \})) = \hbox{\rm deg}_\infty L_+, \hbox{\em if } n \ \hbox{\em is even},$$
$$ \chi ( {\rm Lk}^\infty ( \{f \ge 0 \})) = 2-\hbox{\rm deg}_\infty L_+, \hbox{\em if } n \ \hbox{\em is odd},$$
$$ \chi ( {\rm Lk}^\infty ( \{f = 0 \})) =\hbox{\rm deg}_\infty L_- + \hbox{\rm deg}_\infty L_+, \hbox{\em if } n \ \hbox{\em is even},$$
$$ \chi ( {\rm Lk}^\infty ( \{f  = 0 \})) = \hbox{\rm deg}_\infty L_- -\hbox{\rm deg}_\infty L_+, \hbox{\em if } n \ \hbox{\em is odd}.$$
\end{corollary}

$\hfill \Box$

\underline{Application:} Let $X \subset \mathbb{R}^n$ be a closed semi-algebraic set. There exists a $C^2$-semi-algebraic function $f : \mathbb{R}^n \rightarrow \mathbb{R}$ such that $X=f^{-1}(0)$ and $f \ge 0$. By a change of origin if necessary, we can assume that $f(0)>1$.
In this case, $G_+ >0$ and so ${\rm deg}_\infty L_+ =0$.
\begin{corollary}
We have $\chi({\rm Lk}^\infty(X))={\rm deg}_\infty L_-$.
\end{corollary}

Next we focus on the case when the $C^2$ semi-algebraic function is semi-tame.
\begin{definition}
Let $f : \mathbb{R}^n \rightarrow \mathbb{R}$ be a $C^2$-semi-algebraic function. We say that $f$ is semi-tame if for any sequence of points $(x_n)_{n \in \mathbb{N}}$ in $\Gamma_{f,\omega}$ such that $\vert x_n \vert \rightarrow + \infty$, we have $\lim_{n \rightarrow + \infty} \vert f(x_n) \vert =0$ or $+\infty$.
\end{definition}

\begin{lemma}
If $f : \mathbb{R}^n \rightarrow \mathbb{R}$ is semi-tame then $(\nabla f)^{-1}(0) \setminus \{f=0 \}$ is compact.
\end{lemma}
\proof Let $\alpha \not= 0$ be a critical value of $f$. Let $x \in f^{-1}(\alpha) \cap ( \nabla f)^{-1}(0)$, $x$ belongs to $\Gamma_{f,\omega}$. If $f^{-1}(\alpha) \cap \nabla f^{-1}(0)$ is not compact then there exists a sequence of points $(x_n)_{n \in \mathbb{N}}$ such that $\vert x_n \vert \rightarrow + \infty$, $f(x_n)=\alpha$ and $\nabla f (x_n)=0$. But $x_n \in \Gamma_{f,\omega}$, this contradicts the semi-tameness of $f$. \endproof

\begin{lemma}
The functions $g_-$ and $g_+$ are semi-tame.
\end{lemma}
\proof We know that $\Gamma_{g_-,\omega}=\Gamma_{f,\omega}$. Let $(x_n)_{n \in  \mathbb{N}}$ be a sequence of points in $\Gamma_{g_-,\omega}$ such that $\vert x_n \vert \rightarrow + \infty$. Then $x_n \in \Gamma_{f,\omega}$ and $\vert x_n \vert \rightarrow + \infty$ so $\lim_{n \rightarrow + \infty} \vert f(x_n) \vert =0$ or $+\infty$. But $g_-(x_n)=f(x_n)-\frac{1}{\omega(x_n)^k}$ so 
$$\lim_{n \rightarrow + \infty} \vert g_-(x_n) \vert =0 \hbox{ or }+\infty .$$ \endproof

\begin{corollary}
The sets $(\nabla g_-)^{-1}(0)$ and $(\nabla g_+)^{-1}(0)$ are compact.
\end{corollary}
\proof The function $g_-$ is semi-tame so $(\nabla g_-)^{-1}(0) \setminus \{g=0 \}$ is compact. We know that $(\nabla g_-)^{-1}(0) \cap \{g=0 \}$ is compact. \endproof

In \cite{Du3}, the author introduced three sets $\Lambda_h^\le$, $\Lambda_h^=$ and $\Lambda_h^\ge$ for a $C^2$ semi-algebraic function $h$. They are defined as follows: for $* \in \{\le,=,\ge \}$,
$$\displaylines{
\quad \Lambda^*_h = \Big\{ \alpha \in \mathbb{R} \ \vert \ \beta \mapsto \chi \big( \hbox{Lk}^\infty(X \cap \{ h * \beta \}) \big) \hbox{ is not constant } \hfill \cr
\hfill \hbox{ in a neighborhood of } \alpha \Big\}. \quad \cr
}$$
By \cite{Du3}, Lemma 3.12, we have  $\Lambda_{g_-}^\le \subset \{0\}$, $\Lambda_{g_-}^\ge \subset \{0\}$ and $\Lambda_{g_-}^= \subset \{0\}$ since $g_-$ is semi-tame.
\begin{proposition}
Let $\alpha_- \in ]-\infty,0[$ and $\alpha_+ \in ]0,+\infty[$. We have 
$$\chi ( {\rm Lk}^\infty ( \{g_- \le \alpha_- \})) + \chi ( {\rm Lk}^\infty ( \{g_- \le \alpha_+ \})) -\chi ( {\rm Lk}^\infty ( \{g_- \le 0  \}))= $$ $$1-{\rm deg}_\infty \nabla g_-.$$
\end{proposition}
\proof See \cite{Du3}, Theorem 3.16. \endproof

\begin{lemma}
Let $\alpha \not= 0$. We have 
$$\chi ( {\rm Lk}^\infty( \{ g_- \le \alpha \} ))= \chi ( {\rm Lk}^\infty( \{ f \le \alpha \} )).$$
\end{lemma}
\proof We assume first that $\alpha <0$. Furthermore, we can suppose that $\frac{1}{2} \alpha + 1 <0$ because the function $\beta \mapsto \chi \big( {\rm Lk}^\infty(X \cap \{ f \le \beta \}) \big)$ is constant on $]-\infty,0[$ by the semi-tameness of $f$. The function $g_-$ is semi-tame so if $(x_n)_{n \in \mathbb{N}}$ is a sequence of points in $\Gamma_{g_-,\omega}$ such that $\vert x_n \vert \rightarrow + \infty$ and $g(x_n) \le \frac{1}{2} \alpha $ then $\lim_{n \rightarrow + \infty} g_-(x_n) = -\infty$. Since $\Gamma_{f,\omega}=\Gamma_{g_-,\omega}$ and $f$ is semi-tame, we also have $\lim_{n \rightarrow + \infty} f(x_n) = -\infty$ because $f(x_n) \le \frac{1}{2} \alpha + 1 <0$. 
So there exits $R_0 \gg 1$ such that for all $R \ge R_0$, $f(x) \le 2 \alpha$ and $g(x) \le 2 \alpha$ for $x \in \tilde{S}_R \cap \Gamma_{g_-,\omega} \cap \{ g_- \le \frac{1}{2} \alpha  \}$. 
So ${\rm Lk}^\infty (\{ g_- \le \alpha \}) $ is homeomorphic to  $\{ g_- \le \alpha \} \cap \tilde{S}_R$ and ${\rm Lk}^\infty (\{f \le \alpha \}) $ is homeomorphic to  $\{ f \le \alpha \} \cap \tilde{S}_R$ for any $R \ge R_0$. 
But $\{ g_- \le \alpha \} \cap \tilde{S}_R =  \{ f \le \alpha + \frac{1}{R^k} \}) \cap \tilde{S}_R $. Since $\tilde{S}_R \cap \Gamma_{f,\omega} \cap \{ \alpha \le f \le \alpha +  \frac{1}{R^k} \}$ is empty, the set 
$\{ f \le \alpha + \frac{1}{R^k} \} \cap \tilde{S}_R$ is homeomorphic to $\{ f  \le \alpha \} \cap \tilde{S}_R$. 

If $\alpha >0$ the proof is the same, replacing $\{ g \le \frac{1}{2} \alpha \}$ with $\{ g \ge \frac{1}{2} \alpha \}$ and taking $R$ such that $\alpha + \frac{1}{R^k}  < 2 \alpha$. \endproof

\begin{corollary}
Let $\alpha_- \in ]-\infty,0[$ and $\alpha_+ \in ]0,+\infty[$. We have 
$$\chi ( {\rm Lk}^\infty ( \{f \le \alpha_- \})) + \chi ( {\rm Lk}^\infty ( \{f \le \alpha_+ \})) -\chi ( {\rm Lk}^\infty ( \{f \le 0  \}))= $$ $$1-{\rm deg}_\infty \nabla g_-.$$
\end{corollary}

$\hfill \Box$

Similarly, we can state:
\begin{corollary}
Let $\alpha_- \in ]-\infty,0[$ and $\alpha_+ \in ]0,+\infty[$. We have 
$$\chi ( {\rm Lk}^\infty ( \{f \ge \alpha_- \})) + \chi ( {\rm Lk}^\infty ( \{f \ge \alpha_+ \})) -\chi ( {\rm Lk}^\infty ( \{f \ge 0  \}))= $$ $$1-(-1)^n{\rm deg}_\infty \nabla g_+.$$
\end{corollary}

$\hfill \Box$

\begin{corollary}
Let $\alpha_- \in ]-\infty,0[$ and $\alpha_+ \in ]0,+\infty[$. We have 
$$\chi ( {\rm Lk}^\infty ( \{f= \alpha_- \})) + \chi ( {\rm Lk}^\infty ( \{f = \alpha_+ \})) -\chi ( {\rm Lk}^\infty ( \{f = 0  \}))= $$ $$2-\chi(S^{n-1}) - ( {\rm deg}_\infty \nabla g_- +(-1)^n{\rm deg}_\infty \nabla g_+).$$
\end{corollary}
\proof Use the previous two corollaries and the Mayer-Vietoris sequence. \endproof

\underline{Application:} Let $X \subset \mathbb{R}^n$ be a closed semi-algebraic set. There exists a $C^2$-semi-algebraic function $f : \mathbb{R}^n \rightarrow \mathbb{R}$ such that $X=f^{-1}(0)$ and $f \ge 0$. By a change of origin if necessary, we can assume that $f(0)>1$.
\begin{lemma}
There exists a $C^2$-semi-algebraic function $\Phi : \mathbb{R}^n \rightarrow \mathbb{R}$ such that 

\begin{enumerate}

\item  $X=\Phi^{-1}(0)$ and $\Phi \ge 0,$
\vspace{0.2cm}

\item  for any sequence of points $(x_n)_{n \in \mathbb{N}}$ in $\Gamma_{\Phi,\omega} \setminus X$ such that $\vert x_n \vert \rightarrow + \infty$, we have $\lim_{n \rightarrow + \infty} \vert \Phi(x_n) \vert=+\infty .$
\vspace{0.2cm}

\end{enumerate}

Moreover one can take $\Phi = \omega^K f$ with $K \gg 1$.
\end{lemma}
\proof Same proof as Lemma \ref{first} (see also \cite{Du2}, Proposition 7.1). \endproof

\begin{corollary}
The function $\Phi$ is semi-tame.
\end{corollary}

$\hfill \Box$

Let us remark that 
\begin{enumerate}
\item $\{ \Phi \le \alpha_-\}$  and $\{ \Phi = \alpha_-\}$ are empty,
\item $\{ \Phi \ge \alpha_-\}=\{ \Phi \ge 0\}$,
\item $\{ \Phi \le 0 \} = \{ \Phi=0\}$,
\item If $\alpha_+>0$ is small enough then ${\rm Lk}^\infty ( \{ \Phi \le \alpha_+ \})$ retracts by deformation to ${\rm Lk}^\infty ( \{ \Phi= 0 \})$.
\end{enumerate}
In this situation, the above corollaries applied to $\Phi$ become the following corollaries.
\begin{corollary}
We have ${\rm deg}_\infty \nabla g_- =1$.
\end{corollary}

$\hfill \Box$

\begin{corollary}
 We have 
$$ \chi ( {\rm Lk}^\infty ( \{\Phi \ge \alpha_+ \})) = 1-(-1)^n{\rm deg}_\infty \nabla g_+.$$
\end{corollary}

$\hfill \Box$

\begin{corollary}
We have 
$$ \chi ( {\rm Lk}^\infty ( \{\Phi = \alpha_+ \})) -\chi ( {\rm Lk}^\infty ( X ))= 1-\chi(S^{n-1}) - (-1)^n{\rm deg}_\infty \nabla g_+.$$
\end{corollary}

$\hfill \Box$

Here we can state the global version of Szafraniec's theorem \cite{Sz1}.
\begin{corollary}\label{GlobalSza}
We have 
$$\chi ( {\rm Lk}^\infty ( X ))= 1 - {\rm deg}_\infty \nabla g_+.$$
\end{corollary}
\proof If $n$ is odd, ${\rm Lk}^\infty (\{ \Phi= \alpha_+ \})$ is a compact odd-dimensional manifold and $\chi(S^{n-1})=2$.

If $n$ is even, $\chi(S^{n-1})=0$. Moreover 
$$\chi( {\rm Lk}^\infty (\{ \Phi= \alpha_+ \})) = 2 \chi ({\rm Lk}^\infty (\{ \Phi \ge  \alpha_+ \}))= 2 -2 {\rm deg}_\infty \nabla g_+.$$
\endproof

\subsubsection{Applications for the global Milnor fiber on the spheres}

Let $F=(f_{1}, \cdots, f_{p}):\mathbb{R}^{n}\to \mathbb{R}^{p}$ be a $C^{2}$ semi-algebraic mapping such that Conditions (A) and (B) hold. In \cite{DAAC} Corollary 4.3, the authors proved a global version of Corollary \ref{milnorformula}.
\begin{proposition}
Let $j \in \{1,\ldots,p\}$. If $n$ is even, then we have $\chi({\rm Lk}^{\infty}(f_j^{-1}(0))= 2 \chi(\mathcal{M}_{\frac{F}{\|F\|}}^{S})$ and if $n$ is odd, then we have $\chi({\rm Lk}^{\infty}(f_j^{-1}(0))=2-2\chi(\mathcal{M}_{\frac{F}{\|F\|}}^{S})$.
\end{proposition}

$\hfill \Box$

Applying this result and Corollary \ref{GlobalSza}, we can state a Poincar\'e-Hopf formula for $\mathcal{M}_{\frac{F}{\|F\|}}^{S}$.
\begin{proposition}
Given $F=(f_1,\ldots,f_p): (\mathbb{R}^n,0) \rightarrow (\mathbb{R}^p,0),$ $1\leq p \leq n-1,$ a $C^2$ semi-algebraic mapping such that  Conditions (A) and (B) hold. Assume for instance that $f_{1}(0)>1.$
There exist two integers $K,k >0$ such that if $g_+(x)=\omega(x)^K f_{1}^{2}(x)+ \frac{1}{\omega(x)^k}$ then we have, 

\begin{enumerate}
\item [(i)] if $n$ is even, then $\chi(\mathcal{M}_{\frac{F}{\|F\|}}^{S})=\frac{1}{2}(1-\deg_{\infty} \nabla g_+);$
\item [(ii)] if $n$ is odd, then $\chi(\mathcal{M}_{\frac{F}{\|F\|}}^{S})=\frac{1}{2}(1+\deg_{\infty} \nabla g_+)$.
\end{enumerate}
\end{proposition}

The authors would like to acknowledge Professor Takashi Nishimura for invite them to submit a paper to RIMS proceedings. Thanks !

\end{document}